\documentclass[11pt,a4paper]{article}

\usepackage{amsfonts}
\usepackage{amsmath}

\usepackage{amsthm,amsfonts,color}
\usepackage{amssymb}
\usepackage{graphicx}
\usepackage{setspace}
\usepackage{mathrsfs}
\usepackage{graphicx}
\usepackage{calc}
\usepackage{tikz}

\newtheorem{theorem}{Theorem}
\newtheorem{conjecture}{Conjecture}

\newtheorem{proposition}{Proposition}

\newtheorem{claim}{Claim}
\newtheorem{corollary}{Corollary}

\tikzstyle{vertex}=[circle, draw, inner sep=2pt, minimum size=6pt]
\tikzstyle{filledvertex}=[circle, draw, fill, inner sep=2pt, minimum size=6pt]
\tikzstyle{directed}=[postaction={decorate,
    decoration={markings,mark=at position 0.5 with {\arrow{stealth}}}
}]

\newcommand{\vertex}{\node[vertex]}

\usepackage{geometry}
\begin{document}
\begin{spacing}{1.10}

\title{Connected $k$-factors in bipartite graphs
\thanks{The first author is supported by NSFC (Nos. 11601430 and 11671320)
and China Postdoctoral Science Foundation (No. 2016M590969), and the
second author is supported by NSFC (No. 11601429) and the
Fundamental Research Funds for the Central Universities (No.
3102018ZY035).}}

\author{\quad Yandong Bai,
\quad Binlong Li \thanks{Corresponding author. E-mail addresses: bai@nwpu.edu.cn (Y. Bai), binlongli@nwpu.edu.cn (B. Li).}\\[2mm]
\small Department of Applied Mathematics, Northwestern Polytechnical University, \\
\small Xi'an 710129, China}
\date{\today}
\maketitle

\begin{abstract}
Let $k,l$ be two positive integers.
An $S_{k,l}$ is a graph obtained from disjoint $K_{1,k}$ and $K_{1,l}$ by adding an edge
between the $k$-degree vertex in $K_{1,k}$ and the $l$-degree vertex
in $K_{1,l}$. An {\em $S_{k,l}$-free} graph is a graph containing no
induced subgraph isomorphic to $S_{k,l}$. In this note, we show
that, for any positive integers $k,l$ with $2\leqslant k\leqslant l$, there
exists a constant $c=c(k,l)$ such that every connected balanced
$S_{k,l}$-free bipartite graph with minimum degree at least $c$
contains a connected $k$-factor.
\medskip

\end{abstract}

\section{Introduction}

In this paper,
we only consider simple graphs, i.e., without loops or multiple edges.
For terminology and notation not defined here,
we refer the reader to \cite{BM2008}.

Let $k,l$ be two positive integers.
An $S_{k,l}$ is a graph obtained from disjoint $K_{1,k}$ and $K_{1,l}$
by adding an edge between the $k$-degree vertex in $K_{1,k}$ and the $l$-degree vertex in $K_{1,l}$,
one can see two simple examples $S_{2,3}$ and $S_{3,3}$ in Figure \ref{figure: S kl}.
Call a graph {\em $F$-free} if it contains no induced subgraph isomorphic to $F$.
A {\em $k$-factor} of a graph is a $k$-regular spanning subgraph of the graph.
Note that a connected 2-factor is a Hamilton cycle,
i.e., a cycle containing all vertices of the graph.

Let $G=(X,Y;E)$ be a bipartite graph with bipartition $(X,Y)$. We
call $G$ {\em balanced} if its two partite sets have the same
cardinality, i.e., $|X|=|Y|$. It is not difficult to see that if a
bipartite graph contains a $k$-factor then it is balanced. Thus
while considering the existence of $k$-factors in a bipartite graph
we will always assume that the bipartite graph is balanced;
similarly, while considering the existence of connected $k$-factors
we will always assume that the bipartite graph is connected. For a
subgraph $H$ of a graph $G$, we use $N_{H}(v)$ to denote the set of
neighbors of $v$ in $H$ and use $d_{H}(v)$ to denote the degree of
$v$ in $H$, i.e., the cardinality of $N_{H}(v)$. For a subset
$A\subseteq V(G)$, we use $N_{G}(A)$ to denote the set of neighbors of $v\in A$ in $G$, i.e.,
$$
N_{G}(A)=\bigcup_{v\in A}N_{G}(v).
$$

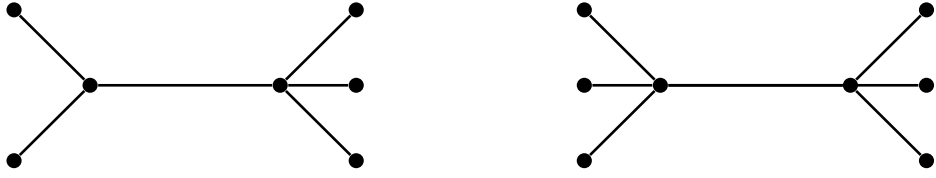
\begin{figure}[ht]
\begin{center}
\begin{tikzpicture}
\tikzstyle{vertex}=[circle,inner sep=2pt, minimum size=0.1pt]



\vertex (a)[fill] at (-6,-1)[]{};
\vertex (b)[fill] at (-6,1)[]{};
\vertex (c)[fill] at (-5,0)[]{};
\vertex (d)[fill] at (-2.5,0)[]{};
\vertex (e)[fill] at (-1.5,-1)[]{};
\vertex (f)[fill] at (-1.5,1)[]{};

\vertex (g)[fill] at (6,-1)[]{};
\vertex (h)[fill] at (6,1)[]{};
\vertex (i)[fill] at (5,0)[]{};
\vertex (j)[fill] at (2.5,0)[]{};
\vertex (k)[fill] at (1.5,-1)[]{};
\vertex (l)[fill] at (1.5,1)[]{};
\vertex (m)[fill] at (6,0)[]{};

\vertex (o)[fill] at (-1.5,0)[]{};
\vertex (r)[fill] at (1.5,0)[]{};

\draw [line width=1.0pt] (a)--(c);
\draw [line width=1.0pt] (b)--(c);
\draw [line width=1.0pt] (c)--(d);
\draw [line width=1.0pt] (d)--(e);
\draw [line width=1.0pt] (d)--(f);
\draw [line width=1.0pt] (d)--(o);

\draw [line width=1.0pt] (i)--(g);
\draw [line width=1.0pt] (i)--(h);
\draw [line width=1.0pt] (i)--(j);
\draw [line width=1.0pt] (j)--(k);
\draw [line width=1.0pt] (j)--(l);
\draw [line width=1.0pt] (j)--(m);
\draw [line width=1.0pt] (j)--(r);

\end{tikzpicture}
\end{center}
\caption{An $S_{2,3}$ and an $S_{3,3}$.}
\label{figure: S kl}
\end{figure}

Petersen \cite{Petersen1891} was regarded as the first to
investigate graphs with $k$-factors. Hall \cite{Hall1935} and Rado
\cite{Rado1949} were among the first ones to study $k$-factors in
bipartite graphs. A milestone result on this topic was obtained by
Tutte \cite{Tutte1952}  in 1952 who presented a necessary and
sufficient condition for a graph to contain a $k$-factor. After
that, finding various sufficient conditions for the existence of
(connected) $k$-factors has been widely studied. For more results on
$k$-factors in graphs, we refer the reader to the survey
\cite{AK2010} by Akiyama and Kano, or to
\cite{EJKS1985,Katerinis1985,Nam2010,Nishimura1992,Nishimura1989,Tutte1947,Tutte1952}.
For more results on $k$-factors in bipartite graphs, we refer the
reader to \cite{EOK1988,Katerinis1987,Ore1957,Ore1959,Rado1949}.

In this note,
we consider connected $k$-factors in bipartite graphs under forbidden subgraph conditions.
Note that any graph with more than 2 vertices contains no connected 1-factor.
Thus we consider the case of $k,l\geqslant 2$ in the following result.

\begin{theorem}\label{thm: connected k-factor}
For any positive integers $k,l$ with $2\leqslant k\leqslant l$, there exists a
constant $c=c(k,l)$ such that every connected balanced
$S_{k,l}$-free bipartite graph with minimum degree at least $c$
contains a connected $k$-factor.
In particular,
$$
c=\max\left\{(k^{3}+1)(2l-2k-1)+l,~2(k^{2}-k+l)\right\}
$$
will do.
\end{theorem}

\begin{theorem}\label{thm: k-factor}
For any positive integers $k,l,m\geqslant 1$, there exists a constant
$c'=c'(k,l,m)$ such that every balanced $S_{k,l}$-free bipartite
graph with minimum degree at least $c'$ contains an $m$-factor. In
particular, $c'=2(\max\{k,l,m\})^2$ will do.
\end{theorem}

For the existence of connected 2-factors, i.e.,
Hamilton cycles, we obtain the following result.

\begin{theorem}\label{thm: HC in S13 free}
Every connected balanced $S_{1,3}$-free bipartite graph with
minimum degree at least $4$ contains a Hamilton cycle.
\end{theorem}

Tutte \cite{Tutte1971} conjectured in 1971 that every 3-connected 3-regular bipartite graph contains a Hamilton cycle.
The conjecture was disproved by Horton \cite{Horton1982},
one can see \cite{EH1983} for more counterexamples.
Note that every 3-regular bipartite graph is $S_{1,3}$-free.
Thus the degree condition in Theorem \ref{thm: HC in S13 free} is sharp.

\section{Proofs of Theorems \ref{thm: connected k-factor}, \ref{thm: k-factor} and \ref{thm: HC in S13 free}}

Before to proceed our proofs, we introduce some necessary notations
and a useful result. Let $f$ be a nonnegative integer-valued
function on $V(G)$. An {\em $f$-factor} of $G$ is a spanning
subgraph $F$ of $G$ such that $d_{F}(v)=f(v)$ for all $v\in V(G)$.
Note that a $k$-factor is an $f$-factor with $f(v)=k$ for all $v\in
V(G)$. The following useful theorem on the existence of $f$-factors
in bipartite graphs was first proved by Ore in \cite{Ore1957}. One
can also find a necessary and sufficient condition with different
form by Tutte in \cite{Tutte1952}.

\begin{theorem}[Ore-Ryser Theorem]\label{thm: ore-ryser}
A bipartite graph $G=(X,Y;E)$ contains an $f$-factor if and only if
$\sum_{x\in X}f(x)=\sum_{y\in Y}f(y)$ and for each subset
$A\subseteq X$,
$$
\sum_{x\in A}f(x)\leqslant \sum_{y\in N_{G}(A)}\min\left\{f(y),d_{A}(y)\right\}.
$$
\end{theorem}

\begin{proof}[\bf{Proof of Theorem \ref{thm: connected k-factor}}]
We first show
that $G$ contains a $k$-factor and then show that there exists a
connected one. Through the proof, we assume that the minimum degree
$\delta(G)$ of $G$ satisfies
$$
\delta(G)\geqslant\max\left\{(k^{3}+1)(2l-2k-1)+l,~2(k^{2}-k+l)\right\}.
$$

\medskip
\noindent
\emph{\textbf{Step 1.} $G$ contains a $k$-factor.}
\medskip

Assume the opposite that $G$ contains no $k$-factor. By Theorem
\ref{thm: ore-ryser}, there exists $A\subseteq X$ such that
\begin{equation}\label{equation: A}
k|A|>\sum_{y\in N_{G}(A)}\min\{k,d_{A}(y)\}.
\end{equation}
We choose such an $A$ with minimum cardinality. Let $A_{0}=A$ and,
for $i\geqslant 1$, let
$$
A_{i}=\{v\in V(G): ~d(v,A)=i\},
$$
where $d(v,A)$ denotes the distance between $v$ and $A$. Now we let
$A_{0}^{0}=A_{0}$, $A_{0}^{1}=A_{0}^{2}=\emptyset$, and for
$i=1,2,\ldots$, set
\begin{equation*}
\begin{split}
A_{i}^{0}   &=\{v\in A_{i}:~d_{A_{i-1}}(v)<k\},\\
A_{i}^{1}   &=\{v\in A_{i}\backslash A_{i}^{0}:~N_{A_{i-1}}(v)\subseteq A_{i-1}^{0}\},\\
A_{i}^{2}   &=A_{i}\backslash (A_{i}^{0}\cup A_{i}^{1}).
\end{split}
\end{equation*}
Clearly $|A_{i}|=|A_{i}^{0}|+|A_{i}^{1}|+|A_{i}^{2}|$.
Denote by
$A_i^{01}=A_i^0\cup A_i^1$ and $A_i^{12}=A_i^1\cup A_i^2$.
We remark that some of $A_{i}^{0},A_{i}^{1},A_{i}^{2}$ may be empty. Recall
that $A_{0}\subseteq X$. We have
\begin{equation}
X=\bigcup_{i\text{ is even}}A_i,~~Y=\bigcup_{i\text{ is odd}}A_i;
\end{equation}
We will get a contradiction by showing that $|X|>|Y|$.

\begin{claim}\label{claim: one set larger than another one}
$|A_{0}^{0}|>|A_{1}^{1}|$ and $|A_{i}^{0}|\geqslant |A_{i+1}^{1}|$ for each $i\geqslant 1$.
\end{claim}

\begin{proof}
By (\ref{equation: A}),
$$
k|A_{0}^{0}|=k|A_{0}|>\sum_{y\in N_{G}(A_{0})}\min\left\{k,d_{A_{0}}(y)\right\}=k|A_{1}^{1}|+\sum_{y\in A_{1}^{0}}d_{A_{0}}(y)\geqslant k|A_{1}^{1}|.
$$
This implies $|A_{0}^{0}|>|A_{1}^{1}|$. We will show the second
assertion in the following.

We first claim that for every vertex $v\in A_i^0$, $i\geqslant 1$,
$N_{A_{i-1}}(v)\subseteq A_{i-1}^0$.
Assume the opposite that $v$ has
a neighbor $u\in A_{i-1}^{12}$. Then $i\geqslant 2$ and the vertex $u$ has $k$
neighbors in $A_{i-2}$.
Since $v\in A_i^0$,
we have $d_{A_{i+1}}(v)\geqslant\delta(G)-(k-1)\geqslant l$ (note that $\delta(G)\geqslant
k+l-1$). Then the vertices $u,v$, together with $k$ vertices in
$N_{A_{i-2}}(u)$ and $l$ vertices in $N_{A_{i+1}}(v)$, induce an
$S_{k,l}$ of $G$, a contradiction. Thus, as we claimed,
$N_{A_{i-1}}(v)\subseteq A_{i-1}^0$ for every vertex $v\in A_i^0$.
Together with the definition of $A_i^1$, we have
\begin{equation}\label{Eq3}
N_{A_{i-1}^{12}}(v)=\emptyset, \mbox{ for every }v\in A_i^{01},~
i\geqslant 1.
\end{equation}

Next we show that for every vertex $x\in A_0^0$, $d_{A_1^0}(x)\leqslant k-1$.
Assume the opposite that $d_{A_1^0}(x)\geqslant k$ for some $x\in A_{0}^{0}$.
Let $A'=A\backslash\{x\}$.
Then
$$
\sum_{y\in N_{G}(A')}\min\{k,d_{A'}(y)\}
\leqslant
k|A_1^1|+\sum_{y\in A_1^0}d_A(y)-k
=
\sum_{y\in N_{G}(A)}\min\{k,d_{A}(y)\}-k<k|A|-k=k|A'|
$$
and
$$
k|A'|>\sum_{y\in N_{G}(A')}\min\{k,d_{A'}(y)\},
$$
contradicting the minimality of $A$.
Thus
\begin{equation}\label{Eq4}
d_{A_1^0}(x)\leqslant k-1, \mbox{ for every }x\in A_0.
\end{equation}

Now we show that for every vertex $v\in A_i^0$, $i\geqslant 1$,
$d_{A_{i+1}^{01}}(v)\leqslant k-1$.
Assume that the opposite assertion holds.
Let $v\in A_i^0$
such that $d_{A_{i+1}^{01}}(v)\geqslant k$ and $i$ is as small as possible.
Let $u$ be a neighbor of $v$ in $A_{i-1}$.
By (\ref{Eq3}),
$u\in A_{i-1}^0$.
By (\ref{Eq4}) and the choice of $i$, we have
$d_{A_i^0}(u)\leqslant k-1$.
It follows that
$d_{A_{i-2}}(u)+d_{A_i^{12}}(u)\geqslant l$.
By (\ref{Eq3}),
$E(A_i^{12},A_{i+1}^{01})=\emptyset$. Thus $u,v$, together with $k$
vertices in $N_{A_{i+1}^{01}}(v)$ and $l$ vertices in
$N_{A_{i-2}}(u)\cup N_{A_i^{12}}(u)$ induce an $S_{k,l}$ of $G$, a
contradiction.
This implies that
\begin{equation}\label{Eq5}
d_{A_{i+1}^{01}}(v)\leqslant k-1, \mbox{ for every }v\in A_i^0,~ i\geqslant 1.
\end{equation}
Consequently, every vertex in $A_i^0$ has at most $k-1$ neighbors in
$A_{i+1}^1$ for each $i\geqslant 1$.
If $A_{i}^{0}\neq \emptyset$,
then, since every vertex in $A_{i+1}^1$
has at least $k$ neighbors in $A_i^0$,
we have $|A_{i}^{0}|>|A_{i+1}^{1}|$.
If $A_{i}^{0}=\emptyset$,
then clearly $A_{i+1}^{1}=\emptyset$ and $|A_{i}^{0}|\geqslant |A_{i+1}^{1}|$.
So we conclude that
$|A_i^0|\geqslant |A_{i+1}^1|$ for each $i\geqslant 1$.
\end{proof}

\begin{claim}\label{claim: one set not smaller than other two}
$|A_{i}^{2}|\geqslant |A_{i-1}^{0}|+|A_{i+1}^{2}|$ for each $i\geqslant 2$.
\end{claim}

\begin{proof}
We first claim that for every vertex $x\in A_2^2$,
$d_{A_1^0}(x)+d_{A_3}(x)\leqslant k(k-1)+l-1$. Assume the opposite that there exists some $x\in A_{2}^{2}$ which contradicts the assertion.
Let $S$ be a set of $k(k-1)+l$ vertices in $N_{A_1^0}(x)\cup N_{A_3}(x)$.
Let $y$ be a neighbor of $x$ in $A_1^{1}$ (note that
$A_1^{2}=\emptyset$ and such a $y$ exists) and let $T$ be a set of $k$ neighbors of $y$ in
$A_0$. If there are $l$ vertices in $S$ that have no neighbor in $T$,
then we can find an induced $S_{k,l}$ of $G$, a contradiction. Thus
there are at most $l-1$ vertices in $S$ that have no neighbor in $T$
and this implies that
$|E(S,T)|\geqslant |S|-(l-1)=k(k-1)+1$.
It follows that there
exists a vertex $x'\in T$ which has at least $\lceil\frac{k(k-1)+1}{k}\rceil=k$ neighbors in $S$. Since
$x'$ has no neighbor in $A_3$,
we get that $x'$ has at least $k$ neighbors in $A_1^0$,
contradicting (\ref{Eq4}).
Thus as we claimed,
\begin{equation}\label{Eq6}
d_{A_1^0}(x)+d_{A_3}(x)\leqslant k(k-1)+l-1, \mbox{ for every }x\in
A_2^2.
\end{equation}

We next claim that for every vertex
$v\in A_i^2$, $i\geqslant 3$,
$d_{A_{i-1}^{01}}(v)+d_{A_{i+1}}(v)\leqslant k(k-1)+l$.
Assume the opposite that there exists some $v\in A_{i}^{2}$ which contradicts the assertion.
Let $u$ be a neighbor of $v$ in $A_{i-1}^{12}$.
Let $S$ be a set of $k(k-1)+l$ vertices in $N_{A_{i-1}^{01}}(v)\cup
N_{A_{i+1}}(v)\backslash\{u\}$
and let $T$ be a set of $k$ neighbors of
$u$ in $A_{i-2}$.
Similarly as the analysis above,
we can see that there exists a vertex $v'\in T$ which has at least $k$ neighbors in $S$.
By (\ref{Eq3}), $v'\in A_{i-2}^0$ and the vertex $v'$ has at least $k$ neighbors in
$A_{i-1}^{01}$, contradicting (\ref{Eq4}).
Thus as we claimed,
\begin{equation}\label{Eq7}
d_{A_{i-1}^{01}}(v)+d_{A_{i+1}}(v)\leqslant k(k-1)+l, \mbox{ for every
}v\in A_i^2, ~i\geqslant 3.
\end{equation}

By (\ref{Eq5}) and the definition of $A_i^0$, for every vertex $v\in
A_i^0$, $i\geqslant 1$, $d_{A_{i+1}^2}(v)\geqslant \delta(G)-2(k-1)$.
Since $\delta(G)\geqslant (k+2)(k-1)+l$,
we have
\begin{equation}\label{Eq8}
d_{A_i^2}(v)\geqslant k(k-1)+l, \mbox{ for every }v\in A_{i-1}^0, ~i\geqslant 2.
\end{equation}
By (\ref{Eq3}) and (\ref{Eq7}),
for every vertex $v\in A_i^2$, $i\geqslant 3$, $d_{A_{i-1}^2}(v)\geqslant \delta(G)-k(k-1)-l$.
Since $\delta(G)\geqslant 2k(k-1)+2l$,
we have
\begin{equation}\label{Eq9}
d_{A_i^2}(v)\geqslant k(k-1)+l, \mbox{ for every }v\in A_{i+1}^2, ~i\geqslant 2.
\end{equation}

By combining (\ref{Eq6}) and (\ref{Eq7}), we have
\begin{equation*}
d_{A_{i-1}^{0}}(v)+d_{A_{i+1}^2}(v)\leqslant  k(k-1)+l, \mbox{ for every
}v\in A_i^2,~ i\geqslant 2.
\end{equation*}
Together with (\ref{Eq8}) and (\ref{Eq9}),
we can get that
$|A_{i}^{2}|\geqslant |A_{i-1}^{0}|+|A_{i+1}^{2}|$ for each
$i\geqslant 2$.
\end{proof}

Now by Claims \ref{claim: one set larger than another one} and
\ref{claim: one set not smaller than other two}, we get that
$|X|>|Y|$, a contradiction to the assumption that $G$ is balanced.
So $G$ has a $k$-factor.
Note that till now we only need to require
$$
\delta(G)\geqslant \max\left\{k+l-1,(k+2)(k-1)+l,2k(k-1)+2l\right\}=2k(k-1)+2l=2(k^{2}-k+l).
$$

\medskip
\noindent
\emph{\textbf{Step 2.} $G$ contains a connected $k$-factor.}
\medskip

We choose a $k$-factor $F$ of $G$ with minimum number of components.
If $F$ is connected, then we are done. So we assume that $F$ is
disconnected. Since $G$ is connected,
there exists some edge $uv\in E(G)$
such that $u$ and $v$ are contained in distinct components of $F$.
We call such an edge a \emph{link}.
If $uv$ is a link and there
are $u'\in N_F(u)$ and $v'\in N_F(v)$ such that $u'v'\in E(G)$,
then
$$
(E(F)\backslash\{uu',vv'\})\cup\{uv,u'v'\}
$$
induces a $k$-factor
with less components, a contradiction. Thus we have
\begin{equation}\label{Eq10}
E_G(N_F(u),N_F(v))=\emptyset, \mbox{ for every link }uv.
\end{equation}

\begin{claim}\label{claim: number of neighbors in other components}
Let $H_1$ be a component of $F$ and $v\in V(H_1)$. Then
$d_{G-H_1}(v)\leqslant (k^2-k+1)(2l-2k-1)$.
\end{claim}

\begin{proof}
Assume w.l.o.g. that $u\in X$. We construct a graph $G^{*}$ on the
vertex set $(V(G)\backslash V(H_1))\cap Y$. Two vertices in
$V(G^{*})$ are adjacent if and only if they have a common neighbor in the graph $F-H_1$.
Since $F$ is $k$-regular, one can see that
each vertex in $G^{*}$ has at most $k(k-1)$ neighbors in $G^{*}$ and
the chromatic number of $G^{*}$ is at most $k(k-1)+1$. It follows
that $V(G^{*})$ can be partitioned into $k(k-1)+1$ independent sets.

Suppose the opposite that $d_{G-H_1}(v)\geqslant (k^2-k+1)(2l-2k-1)+1$. Then there exists
an independent set $I$ of $G^*$ such that $d_I(v)\geqslant 2l-2k$. Let
$S$ be a set of $2l-2k$ vertices in $N_I(v)$ and let $T=\bigcup_{y\in S}N_F(y)$.
Since $S$ is independent in $G^*$,
each two vertices of $S$ has no common neighbors in $F$, i.e., $|T|=k(2l-2k)$. By
(\ref{Eq10}), $E_G(N_{H_1}(v),T)=\emptyset$.

Let $y$ be an arbitrary vertex in $S$. If there are $l-k$ vertices
in $S\backslash\{y\}$ that are not adjacent to any vertex in
$N_F(y)$, then $\{v,y\}$, together with $k$ vertices in $N_{F}(y)$
and $l$ vertices in $N_{H_{1}}(v)\cup S\backslash \{y\}$, induces an
$S_{k,l}$ of $G$, a contradiction. Thus there are at most $l-k-1$
vertices in $S\backslash\{y\}$ which are not adjacent to any vertex
in $N_F(y)$. It follows that
$$
|E_G(S,N_F(y))|\geqslant
(2l-2k-1)-(l-k-1)+k=l,
$$
and hence $|E_G(S,T)|\geqslant l(2l-2k)$. So there exists a vertex
$y'\in S$ with $d_T(y')\geqslant l$. It follows that $\{v,y'\}$,
together with $k$ vertices in $N_{H_1}(v)$ and $l$ vertices in
$N_T(y')$, induces an $S_{k,l}$ of $G$, a contradiction.
\end{proof}

\begin{claim}\label{claim: number of neighbors in its component}
Let $v_1v_2$ be a link and let $H_1$ be the component of $F$
containing $v_1$. Then $|N_{G}(v_1)\cap V(H_1)|\leqslant
k(k^2-k+1)(2l-2k-1)+l-1$.
\end{claim}

\begin{proof}
Assume that $|N_{G}(v_1)\cap V(H_1)|\geqslant
k(k^2-k+1)(2l-2k-1)+l$. Let $S$ be a set of $k(k^2-k+1)(2l-2k-1)+l$
vertices in $N_{G}(v_1)\cap V(H_1)$ and let $T=N_F(v_2)$. Note that
$|T|=k$ and there always exists an edge between $T$ and each $l$
vertices in $S$. It therefore follows that
$$
|E_G(S,T)|\geqslant k(2l-2k-1)(k^2-k+1)+l-(l-1)=k(2l-2k-1)(k^2-k+1)+1.
$$
Thus there exists one vertex $v$ of $T$ such that
$$
d_{H_1}(v)\geqslant \left\lceil\frac{k(2l-2k-1)(k^2-k+1)+1}{k}\right\rceil\geqslant (2l-2k-1)(k^2-k+1)+1,
$$
contradicting Claim \ref{claim: number of neighbors in other components}.
\end{proof}

From Claims \ref{claim: number of neighbors in other components} and
\ref{claim: number of neighbors in its component}, there exists a
vertex in $G$ with degree at most
$$
(k+1)(k^2-k+1)(2l-2k-1)+l-1=(k^3+1)(2l-2k-1)+l-1<\delta(G),
$$
a contradiction.
The proof of Theorem \ref{thm: connected k-factor} is complete.
\end{proof}

\begin{proof}[\bf{Proof of Theorem \ref{thm: k-factor}}]
Let $t=\max\{k,l,m\}$ and let $G$ be a connected $S_{k,l}$-free
balanced bipartite graph with minimum degree $\delta(G)\geqslant
2t^2$. Since $S_{k,l}$ is an induced subgraph of $S_{t,t}$, we have
that $G$ is $S_{t,t}$-free. From the Step 1 of the proof of Theorem
\ref{thm: connected k-factor}, 
we can see that $G$ contains a $t$-factor $F$. 
Note that $F$ is a $t$-regular bipartite graph and it contains an $s$-factor for all $s\leqslant t$. 
Thus $G$ contains an $m$-factor.
\end{proof}

\begin{proof}[\bf{Proof of Theorem \ref{thm: HC in S13 free}}]

Similar to the proof of Theorem \ref{thm: connected k-factor},
we first show that $G$ contains a 2-factor and then show that there exists a connected one.

Assume the opposite that $G$ contains no 2-factor.
By Theorem \ref{thm: ore-ryser},
there exists $A\subseteq X$ such that
\begin{equation}\label{Eq11}
2|A|>\sum_{y\in N(A)}\min\{2,d_A(y)\}.
\end{equation}
We choose $A$ as small as possible.
For $i\geqslant 0$,
set $A_i=\{v\in V(G): ~d(v,A)=i\}$.
Note that $A_0=A$ and $A_1=N(A)$.
We partite $A_1$ and $A_2$ into two parts.
Let
\begin{equation*}
\begin{split}
A_1^0=\{y\in A_1: d_{A_0}(y)=1\},\  & A_1^1=A_1\backslash A_1^0; \\
A_2^0=\{x\in A_2: d_{A_3}(x)=0\},\  & A_2^1=A_2\backslash A_2^0.
\end{split}
\end{equation*}

From (\ref{Eq11}),
one can easily get that $|A_0|>|A_1^1|$.
For every vertex $v\in A_i$, $i\geqslant 2$,
we claim that $d_{A_{i+1}}(v)\leqslant 2$.
If not, then let $u_1,u_2,u_3$ be three neighbors of $v$ in $A_{i+1}$,
let $u$ be a neighbor of $v$ in $A_{i-1}$,
and let $v_1$ be a neighbor of $u$ in $A_{i-2}$.
Then $\{v,v_1,u,u_1,u_2,u_3\}$
induces an $S_{1,3}$ of $G$, a contradiction. Thus
$$
d_{A_{i+1}}(v)\leqslant 2,~
d_{A_{i-1}}(v)\geqslant 2,~\mbox{for every vertex }v\in A_i,~i\geqslant 2.
$$
It follows that $|A_2^1|\geqslant |A_3|$ and
$|A_i|\geqslant |A_{i+1}|$ for each $i\geqslant 3$.

To get a contradiction,
we will show that $|A_2^0|\geqslant |A_1^0|$ in the
following.

Let $x\in A_0$ that has a neighbor $y$ in $A_1^0$. We claim that
every vertex in $N(x)\backslash\{y\}$ has three neighbors in $A_0$.
Suppose not. Let $A'=A\backslash\{x\}$ and
$$
2|A'|=2|A|-2>\sum_{y\in N(A)}\min\{2,d_A(y)\}-2\geqslant \sum_{y\in
N(A')}\min\{2,d_{A'}(y)\},
$$
contradicting the choice of $A$. Thus
we conclude that every vertex $x\in A_0$ has at most one neighbor in
$A_1^0$, and if $x$ has a neighbor in $A_1^0$, then each of its
other neighbors has at least three neighbors in $A_0$.

Let $y$ be an arbitrary vertex in $A_1^0$. We show that
$N_{A_2}(y)\subseteq A_2^0$. Let $x$ be the (unique) neighbor of $y$
in $A_0$ and let $y_1,y_2,y_3$ be three neighbors of $x$ other than
$y$. It follows that $y_i$ has three neighbors in $A_0$ for $i\in
\{1,2,3\}$. Let $x_1$ be a neighbor of $y$ in $A_2$. Then $x_1y_i\in
E(G)$ for some $i\in \{1,2,3\}$; otherwise $\{x,y,x_1,y_1,y_2,y_3\}$
induces an $S_{1,3}$. If $x_1$ has a neighbor $y'\in A_3$, then
$\{x,y_i,y'\}$, together with three neighbors of $y_i$ in $A_0$,
induces an $S_{1,3}$, a contradiction. Thus we conclude
$N_2(y)\subseteq A_2^0$. Specially, every vertex in $A_1^0$ has at
least three neighbors in $A_2^0$.

Let $x$ be an arbitrary vertex in $A_2^0$. We claim that $x$ has at most three neighbors in $A_1^0$.
Assume the opposite that  $y_1,\ldots,y_4$
are four neighbors in $N_{A_1^0}(x)$. Let $x_1$ be a neighbor of
$y_1$ in $A_0$. Recall that $x_1$ has no neighbors in $A_1^0$ other
than $y_1$. Thus $\{x,x_1,y_1,y_2,y_3,y_4\}$ induces an $S_{1,3}$, a
contradiction. Thus as we claimed, every vertex in $A_2^0$ has at
most three neighbors in $A_1^0$.
It follows that
$|A_2^0|\geqslant |A_1^0|$ and hence $|X|>|Y|$, contradicting the assumption that $G$ is
balanced.

Now we show that $G$ contains a connected 2-factor, i.e., a Hamilton cycle.
Let $F$ be a 2-factor of $G$ with minimum number of components. If
$F$ is connected, then it is a Hamilton cycle. So we assume that $F$
is disconnected. Note that every component of $F$ is a cycle.

\setcounter{claim}{0}
\begin{claim}\label{ClComponents}
For every vertex $u$ of $G$,
$N_{G}(u)$ is contained in at most two
components of $F$.
\end{claim}

\begin{proof}
We finish the proof by contradiction.
Let $H_0,H_1,H_2$ be three components of $F$ such
that $u\in V(H_0)$ and $uv_i\in E(G)$ for $v_i\in V(H_i)$, $i=1,2$.
Let $N_F(u)=\{u',u''\}$ and $v'_i$ be a vertex in $N_F(v_i)$,
$i=1,2$. Clearly $u'v'_i,u''v'_i\notin E(G)$, otherwise we can get a
2-factor of $G$ with less components. If $v_1v'_2\notin E(G)$, then
$\{u,u',u'',v_1,v_2,v'_2\}$ induces an $S_{1,3}$, a contradiction.
Thus we conclude that $v_1v'_2\in E(G)$, and similarly, $v_2v'_1\in
E(G)$. It follows that
$(E(F)\backslash\{v_1v'_1,v_2v'_2\})\cup\{v_1v'_2,v_2v'_1\}$ induces
a 2-factor with less components, a contradiction.
\end{proof}

\begin{claim}\label{ClC4}
Every component of $F$ is a $4$-cycle $C_4$.
\end{claim}

\begin{proof}
Let $H$ be a component of $F$. Since $G$ is connected, there is some
edges between $H$ and $H'$, where $H'$ is a component of $F$ other
than $H$.

We first show that there is a vertex in $H'$ has at least two
neighbors in $H$. Let $u\in V(H')$, $v\in V(H)$ with $uv\in E(G)$.
If $d_H(u)=1$,
then, by Claim \ref{ClComponents} and the fact
$\delta(G)\geqslant 4$,
we get that $u$ has a neighbor $u'\in V(H')\backslash N_F(u)$.
If there is a vertex $v'\in N_F(v)$ which is  nonadjacent to $u'$,
then $\{u,u',v,v'\}\cup N_F(u)$ induces an $S_{1,3}$, a contradiction.
This implies that $v'$ is adjacent to both vertices in $N_F(v)$.
In any case there exists a vertex in $H'$ which has at least two neighbors in $H$.

Now let $u$ be a vertex in $H'$ which has at least two neighbors $v,v'$ in $H$.
Then $E(N_{H'}(u),N_H(v)\cup N_H(v'))=\emptyset$.
Suppose that $H$ is not a $C_4$,
then $N_H(v')\backslash
N_H(v)\neq\emptyset$.
Let $v''$ be a vertex in $N_H(v')\backslash
N_H(v)$. If $vv''\in E(G)$, letting $u'\in N_{H'}(u)$, then
$\{u,u',v,v''\}\cup N_H(v)$ induces an $S_{1,3}$; if $vv''\notin
E(G)$, then $\{u,v,v',v''\}\cup N_{H'}(u)$ induces an $S_{1,3}$,
both yield a contradiction.
Thus we conclude that $H$ is a $C_4$.
\end{proof}

By Claims \ref{ClComponents} and \ref{ClC4}, if there exists an edge $xy$ with $x\in X\cap V(H_1)$ and $y\in Y\cap V(H_2)$, 
then every vertex in $X\cap V(H_1)$ is adjacent to all vertices in $Y\cap V(H_2)$. 
This implies that there is no edge between $Y\cap V(H_1)$ and  $X\cap V(H_2)$; otherwise we can get
a 2-factor with less components.
It therefore follows that $G$ is the double
graph of a cycle, which is Hamiltonian
(The {\em double graph} of a given graph $R$ is constructed by making two copies of $R$,
say $R$ and $R'$,
including the initial edge set of each graph,
and adding edges $uv'$ and $u'v$ for every edge $uv$ of $R$).
The proof of Theorem \ref{thm: HC in S13 free} is complete.
\end{proof}

\section{Concluding remarks}

In 1994,
Flandrin, Fouquet and Li \cite{FFL1994} proposed the following conjecture.

\begin{conjecture}[Flandrin, Fouquet and Li \cite{FFL1994}]\label{conj: FFL1994}
There exists a constant $d$ such that
every connected balanced $S_{3,3}$-free bipartite graph with minimum degree at least $d$ contains a Hamilton cycle.
\end{conjecture}

The only known result for Conjecture \ref{conj: FFL1994} was obtained by Flandrin, Fouquet and Li \cite{FFL1994}
who showed that
every connected balanced $S_{3,3}$-free bipartite graph with minimum degree at least $ \max\{9,\frac{n+14}{6}\}$ contains a Hamilton cycle,
where $n$ is the order of the graph.
Taking (i) $k=2$, $l=3$, (ii) $k=l=3$ in Theorems \ref{thm: connected k-factor} and \ref{thm: k-factor},
we obtain the following two results corresponding to Conjecture \ref{conj: FFL1994}.

\begin{corollary}\label{corollary: HC}
Every connected balanced $S_{2,3}$-free bipartite graph with minimum degree at least 12 contains a Hamilton cycle.
\end{corollary}

\begin{corollary}\label{corollary: 3-factor}
Every connected balanced $S_{3,3}$-free bipartite graph with minimum
degree at least 18 contains a 2-factor and a connected 3-factor.
\end{corollary}

At the end of this note, we give more information on $S_{k,l}$-free bipartite graphs. It is not difficult to show that  a connected bipartite graph is $S_{1,1}$-free if and only if it is a complete bipartite graph. 
We can also show the following proposition. One can compare it with the fact that every connected $K_{1,3}$- and
$Z_1$-free graph is a path, or a cycle, or a complete graph minus the
edges of a matching, where $Z_1$ is the graph obtained from a
triangle by adding a pendant edge.

\begin{proposition}
A connected bipartite graph $G$ is $S_{1,2}$-free if and only if $G$
is a path, or an even cycle, or can be obtained from a complete
bipartite graph by removing a matching.
\end{proposition}

\begin{proof}
The sufficiency is clear.
So we prove the necessity of the assertion.
We assume that $G$ is $S_{1,2}$-free but it is not a path, not an even cycle,  and cannot be
obtained from a complete bipartite graph by removing a matching.

\setcounter{claim}{0}
\begin{claim}\label{ClDegree2}
If $u_1u_2u_3u_4$ is an induced path of $G$,
then $d(u_i)\neq 2$ for some
$i\in \{1,2,3,4\}$.
\end{claim}

\begin{proof}
Assume the opposite that $d(u_i)=2$ for each $i\in \{1,2,3,4\}$.
Let $P=v_0v_1v_2\cdots v_k$ be a longest path with all internal vertices
of degree 2. Then $k\geqslant 5$. Assume w.l.o.g. that
$d(v_0)\leqslant d(v_k)$. If $d(v_1)=d(v_k)=1$, then $G$ is a path,
a contradiction. If $d(v_p)=2$, then $N(v_p)=\{v_{k-1},v_0\}$, since
otherwise there will be a path longer than $P$ with all internal
vertices of degree 2. It follows that $N(v_0)=\{v_1,v_k\}$ and $G$
is a cycle, a contradiction. So we assume that $d(v_k)\geq 3$. Let
$w_1,w_2\in N(v_k)\backslash\{v_{k-1}\}$. Then
$\{v_{k-2},v_{k-1},v_k,w_1,w_2\}$ induces an $S_{1,2}$ of $G$, a contradiction.
\end{proof}

\begin{claim}\label{ClDistance}
The distance of any two vertices in $G$ is at most $4$.
\end{claim}

\begin{proof}
Assume the opposite that there exist two vertices $u,v$ with distance $d(u,v)=5$.
Let $P=v_0v_1\cdots v_5$ be a shortest path between $u=v_0$ and $v=v_5$.
If $v_2$ has a neighbor $w$ other than $v_1,v_3$,
 then either
$wv_0\notin E(G)$ or $wv_4\notin E(G)$;
since otherwise, $d(u,v)\leqslant 3$.
It follows that either $\{v_0,v_1,v_2,v_3,w\}$ or
$\{v_1,v_2,v_3,v_4,w\}$ induces an $S_{1,2}$, a contradiction. Thus
we conclude that $d(v_2)=2$ and similarly $d(v_3)=2$. If $v_1$ has a
neighbor $w'$ other than $v_0,v_3$, then $\{v_0,v_1,v_2,v_3,w'\}$
induces an $S_{1,2}$, a contradiction. Thus we conclude that
$d(v_1)=2$ and similarly $d(v_4)=2$. It follows that $v_1\cdots v_4$
is an induced path with all vertices of degree 2, contradicting
Claim \ref{ClDegree2}.
\end{proof}

Recall that $G$ is not a complete bipartite graph minus the edges of a matching,
then there exist vertices $u,v_1,v_2$ such that
$uv_1,uv_2\notin E(G)$, and $u$ is contained in the partite set different from the one containing $v_1,v_2$.
By Claim \ref{ClDistance}, $d(u,v_i)=3$
for $i\in \{1,2\}$. Let $P_1=a_0a_1a_2a_3$ and $P_2=b_0b_1b_2b_3$ be
two shortest paths between $u$ and $v_1,v_2$, respectively, where $u=a_0=b_0$,
$v_1=a_3$ and $v_2=b_3$. Subject to all above conditions, we take
$P_1,P_2$ such that $E(P_1)\cup E(P_2)$ is as small as possible.

Note that there exists a maximum $j$ such that $a_i=b_i$ for all
$i\leqslant j$ (possibly $j=0$).
By our assumption that $E(P_1)\cup
E(P_2)$ is as small as possible, we have that $\{a_{j+1},\ldots,a_3\}$ and
$\{b_{j+1},\ldots,b_3\}$ are disjoint. If $j=2$, then
$\{a_0,a_1,a_2,a_3,b_3\}$ induces an $S_{1,2}$, a contradiction. If
$j=1$, then $a_2b_3\notin E(G)$;
since otherwise $P'_2=a_0a_1a_2b_3$
is a shortest path between $u$ and $v_2$ such that $E(P_1)\cup
E(P'_2)$ is smaller that $E(P_1)\cup E(P_2)$. It follows that
$\{v_0,v_1,v_2,b_2,b_3\}$ induces an $S_{1,2}$, a contradiction.
Suppose now that $j=0$.

If $a_0$ has a neighbor $w$ other that $a_1,b_1$, then either
$wa_2\notin E(G)$ or $wb_2\notin E(G)$; otherwise we can find a
shortest path $P'_i$ between $u$ and $v_i$, $i\in \{1,2\}$, such that
$E(P'_1)\cup E(P'_2)$ is smaller that $E(P_1)\cup E(P_2)$. It
follows that either $\{a_0,a_1,a_2,b_1,w\}$ or
$\{a_0,a_1,b_1,b_2,w\}$ induces an $S_{1,2}$, a contradiction. Thus
we conclude that $d(a_0)=2$. Suppose that $a_1$ has a neighbor $w'$
other than $a_0,a_3$. Clearly $w'\neq b_2$. We have $b_1w'\in E(G)$;
otherwise $\{a_0,a_1,a_2,b_1,w'\}$ induces an $S_{1,2}$. Thus either
$w'a_3\notin E(G)$ or $w'b_3\notin E(G)$; otherwise we can find a
shortest path $P'_i$ between $u$ and $v_i$, $i=1,2$, such that
$E(P'_1)\cup E(P'_2)$ is smaller that $E(P_1)\cup E(P_2)$. It
follows that either $\{a_0,a_1,a_2,b_1,w'\}$ or
$\{a_0,b_1,b_2,b_3,w'\}$ induces an $S_{1,2}$, a contradiction. This
implies that $d(a_1)=2$ and similarly $d(b_1)=2$.
By Claim \ref{ClDegree2},
the vertex $a_2$ has a neighbor $w''$ other than $a_1,a_3$. It
follows that $\{a_0,a_1,a_2,a_3,w''\}$ induces an $S_{1,2}$, a
contradiction.
\end{proof}

\end{spacing}
\end{document}